\documentclass[12pt]{amsart}
\usepackage{amsmath,amssymb,amsbsy,amsfonts,latexsym,amsopn,amstext,
                                               amsxtra,euscript,amscd,bm}
\usepackage{url}
\usepackage[colorlinks,linkcolor=blue,anchorcolor=blue,citecolor=blue]{hyperref}
\usepackage{color}
\usepackage{float}

\hypersetup{breaklinks=true}

\begin{document}

\newtheorem{theorem}{Theorem}
\newtheorem{lemma}[theorem]{Lemma}
\newtheorem{claim}[theorem]{Claim}
\newtheorem{cor}[theorem]{Corollary}
\newtheorem{prop}[theorem]{Proposition}
\newtheorem{definition}{Definition}
\newtheorem{question}[theorem]{Question}
\newtheorem{remark}[theorem]{Remark}
\newcommand{\hh}{{{\mathrm h}}}

\numberwithin{equation}{section}
\numberwithin{theorem}{section}
\numberwithin{table}{section}

\def\sssum{\mathop{\sum\!\sum\!\sum}}
\def\ssum{\mathop{\sum\ldots \sum}}
\def\iint{\mathop{\int\ldots \int}}

\def\squareforqed{\hbox{\rlap{$\sqcap$}$\sqcup$}}
\def\qed{\ifmmode\squareforqed\else{\unskip\nobreak\hfil
\penalty50\hskip1em\null\nobreak\hfil\squareforqed
\parfillskip=0pt\finalhyphendemerits=0\endgraf}\fi}

\newfont{\teneufm}{eufm10}
\newfont{\seveneufm}{eufm7}
\newfont{\fiveeufm}{eufm5}
%
%
\newfam\eufmfam
     \textfont\eufmfam=\teneufm
\scriptfont\eufmfam=\seveneufm
     \scriptscriptfont\eufmfam=\fiveeufm
%
%
\def\frak#1{{\fam\eufmfam\relax#1}}

\newcommand{\bflambda}{{\boldsymbol{\lambda}}}
\newcommand{\bfmu}{{\boldsymbol{\mu}}}
\newcommand{\bfxi}{{\boldsymbol{\xi}}}
\newcommand{\bfrho}{{\boldsymbol{\rho}}}

\def\fK{Frak K}
\def\fT{Frak{T}}

\def\fA{{Frak A}}
\def\fB{{Frak B}}
\def\fC{{Frak C}}

\def \balpha{\bm{\alpha}}
\def \bbeta{\bm{\beta}}
\def \bgamma{\bm{\gamma}}
\def \blambda{\bm{\lambda}}
\def \bchi{\bm{\chi}}
\def \bphi{\bm{\varphi}}
\def \bpsi{\bm{\psi}}

\def\eqref#1{(\ref{#1})}

\def\vec#1{\mathbf{#1}}


\def\cA{{\mathcal A}}
\def\cB{{\mathcal B}}
\def\cC{{\mathcal C}}
\def\cD{{\mathcal D}}
\def\cE{{\mathcal E}}
\def\cF{{\mathcal F}}
\def\cG{{\mathcal G}}
\def\cH{{\mathcal H}}
\def\cI{{\mathcal I}}
\def\cJ{{\mathcal J}}
\def\cK{{\mathcal K}}
\def\cL{{\mathcal L}}
\def\cM{{\mathcal M}}
\def\cN{{\mathcal N}}
\def\cO{{\mathcal O}}
\def\cP{{\mathcal P}}
\def\cQ{{\mathcal Q}}
\def\cR{{\mathcal R}}
\def\cS{{\mathcal S}}
\def\cT{{\mathcal T}}
\def\cU{{\mathcal U}}
\def\cV{{\mathcal V}}
\def\cW{{\mathcal W}}
\def\cX{{\mathcal X}}
\def\cY{{\mathcal Y}}
\def\cZ{{\mathcal Z}}
\newcommand{\rmod}[1]{\: \mbox{mod} \: #1}

\def\cg{{\mathcal g}}

\def\vr{\mathbf r}

\def\e{{\mathbf{\,e}}}
\def\ep{{\mathbf{\,e}}_p}
\def\em{{\mathbf{\,e}}_m}

\def\Tr{{\mathrm{Tr}}}
\def\Nm{{\mathrm{Nm}}}

 \def\SS{{\mathbf{S}}}

\def\lcm{{\mathrm{lcm}}}
\def\ord{{\mathrm{ord}}}

\def\({\left(}
\def\){\right)}
\def\fl#1{\left\lfloor#1\right\rfloor}
\def\rf#1{\left\lceil#1\right\rceil}

\def\mand{\qquad \mbox{and} \qquad}

\newcommand{\commK}[1]{\marginpar{%
\begin{color}{red}
\vskip-\baselineskip 
\raggedright\footnotesize
\itshape\hrule \smallskip K: #1\par\smallskip\hrule\end{color}}}

\newcommand{\commI}[1]{\marginpar{%
\begin{color}{magenta}
\vskip-\baselineskip 
\raggedright\footnotesize
\itshape\hrule \smallskip I: #1\par\smallskip\hrule\end{color}}}

\newcommand{\commT}[1]{\marginpar{%
\begin{color}{blue}
\vskip-\baselineskip 
\raggedright\footnotesize
\itshape\hrule \smallskip T: #1\par\smallskip\hrule\end{color}}}




\hyphenation{re-pub-lished}

\mathsurround=1pt

\def\bfdefault{b}
\overfullrule=5pt

\def \F{{\mathbb F}}
\def \K{{\mathbb K}}
\def \N{{\mathbb N}}
\def \Z{{\mathbb Z}}
\def \Q{{\mathbb Q}}
\def \R{{\mathbb R}}
\def \C{{\mathbb C}}
\def\Fp{\F_p}
\def \fp{\Fp^*}

\def\Kmn{\cK_p(m,n)}
\def\psmn{\psi_p(m,n)}
\def\SI{\cS_p(\cI)}
\def\SIJ{\cS_p(\cI,\cJ)}
\def\SAIJ{\cS_p(\cA;\cI,\cJ)}
\def\SABIJ{\cS_p(\cA,\cB;\cI,\cJ)}
\def \xbar{\overline x_p}

\title[Divisor problem in arithmetic progressions]{Divisor problem in arithmetic progressions modulo a prime power}

 \author[K. Liu] {Kui Liu}
\address{School of Mathematics and  Statistics, Qingdao University, No.308, Ningxia Road, Shinan, Qingdao, Shandong, 266071, P. R. China}
\email{liukui@qdu.edu.cn}

 \author[I. E. Shparlinski] {Igor E. Shparlinski}

\address{Department of Pure Mathematics, University of New South Wales,
Sydney, NSW 2052, Australia}
\email{igor.shparlinski@unsw.edu.au}

 \author[T. P. Zhang] {Tianping Zhang}

 \thanks{T. P. Zhang is the corresponding author (tpzhang@snnu.edu.cn).}

\address{School of Mathematics and Information Science, Shaanxi Normal University, Xi'an 710062 Shaanxi, P. R. China}
\email{tpzhang@snnu.edu.cn}

\begin{abstract} We obtain an asymptotic formula for the average value of the divisor
function over the integers $n \le x$ in an arithmetic progression $n \equiv a \pmod q$,
where $q=p^k$ for a prime $p\ge 3$ and a sufficiently large integer $k$. In particular, we break
the classical barrier $q \le x^{2/3}$ for such formulas, and generalise a recent result
of R.~Khan (2015), making it uniform in $k$.
 \end{abstract}


\maketitle

\section{Introduction}

\subsection{Background}

For a positive integer $n$, let $d(n)$ be the classical divisor function, which is the number of divisors of $n$. Let $a$ and $q$ be integers with $q\ge 1$ and $\gcd(a,q)=1$. For $X\geq 2$, define
$$
D(X;q,a):=\sum\limits_{\substack{n\leq X\\n\equiv a\bmod q}}d(n).
$$
and also
$$
E(X;q,a):=D(X;q,a)-\frac{1}{\varphi(q)}\sum\limits_{\substack{n\leq X\\\gcd(n,q)=1}}d(n).
$$
In unpublished works, it has been discovered independently by Selberg and Hooley that for any $\varepsilon>0$
there exists some $\delta >0$ such that for a sufficiently large $X$
\begin{equation}
\label{Asymptotic formula}
|E(X;q,a)| \le X^{1-\delta}/q
\end{equation}
holds uniformly for $q\leq X^{2/3-\varepsilon}$. This follows from Weil bound for Klooterman sums,
see~\cite{PoVau}.

When $q$ is large, there are various results on the average bound of $E(X;q,a)$. Fouvry~\cite[Corollary~5]{Fouvry} has studied   the average over $q$ and shown that for any $\varepsilon>0$ there exist some constant $c>0$ such that for a sufficiently large $X$ for any  $a\in\Z$ with $|a| \le \exp(c \sqrt{\log X})$
we have
$$
\sum\limits_{\substack{X^{2/3+\varepsilon}\leq q\leq X^{1-\varepsilon}\\\gcd (q,a)=1}}\left|E(X;q,a)\right|\le X \exp(-c \sqrt{\log X})
$$
Banks, Heath-Brown and Shparlinski~\cite{Banks} have considered the average over $a$ and proved that for any $\varepsilon>0$
there exists some $\delta >0$ such that for a sufficiently large $X$
$$
\sum\limits_{\substack{1\leq a\leq q\\\gcd(a,q)=1}}\left|E(X;q,a)\right|\le X^{1-\delta}
$$
holds uniformly for $q<X^{1-\varepsilon}$. For other examples, see~\cite{Blomer,FouIw,FrIw1,FrIw2,LauZhao}.

Irving~\cite{Irving} first  has broken through the range given by Weil bound
(see~\cite[Corollary~11.12]{IwKow}) for some special individual modulus $q$ and proved that,
for any $\varpi,\varrho>0$ satisfying $246\varpi+18\varrho<1$, there exists some $\delta>0$, depending only on $\varpi$ and $\varrho$ such that~\eqref{Asymptotic formula} holds uniformly for any $x^{\varrho}$-smooth, squarefree moduli $q\leq X^{2/3+\varpi}$.
Khan~\cite{Khan} has considered another important case: the prime power moduli and proved that for a fixed integer $k\geq 7$, there exists some constant $\rho>0$, depending only on $k$, such that~\eqref{Asymptotic formula} holds uniformly for $X^{2/3-\rho}<q<X^{2/3+\rho}$ with $q=p^k$, where $p$ is a sufficiently large prime number.

\subsection{Our results}
In this paper, we focus on the prime power moduli case.

Before we formulate our result we need to recall that  the notations  $U \ll V$ and $U = O(V)$, are equivalent
to $|U|  \le c V)$ for some constant $c>0$. We write $\ll_\rho$ and $O_\rho$ to indicate that
this constant  may depend on the parameter $\rho$.

\begin{theorem}\label{Theorem on divisor problem}
There exist absolute  constants
 $k_0 \ge 1$ and $\sigma>0$ such that
$$
E(X;q,a)\ll \frac{X}{q^{1+\sigma}}
$$
holds uniformly for $q\leq X^{2/3+\sigma}$ with $q=p^k$
for an odd prime $p$  and integer $k\geq k_0$.
\end{theorem}
The key in our proof of Theorem~\ref{Theorem on divisor problem} is the following average estimate for Kloosterman sums
$$
S(n,a;q):=\sideset{}{^*}\sum\limits_{b\bmod q}e\left(\frac{nb+a\bar{b}}{q}\right),
$$
with prime power moduli, where $\gcd(a,q)=1$ and $\sum^*$ means summing over reduced
residue classes. 
The proof borrows from some ideas from~\cite{SteShp1,SteShp2} reworked and
adjusted to the case which is relevant to Kloosterman sums.

\begin{theorem}\label{Theorem on Kloosterman sums}
For any $q^{\lambda}\leq N\le q$ with $\lambda>0$, there exist constants $k_0$ and $\tau>0$, depending only on $\lambda$ such that
$$
\sum\limits_{1\leq n\leq N}S(n,a;q)\ll_{\lambda}Nq^{1/2-\tau}
$$
holds uniformly for any integers $a$ satisfying $\gcd(a,p)=1$ and any $q=p^k$
with $p$ an odd prime, $k\geq k_0$.
\end{theorem}

Using that for any integers $a$, $m$ and $n$ with $\gcd(m,p)=1$ we have
$$
S(mn,a;q) = S(a,mn;q) = S(n,am;q).
$$
 Now we reformulate Theorem~\ref{Theorem on Kloosterman sums} in the form
in which we apply it in the proof of Theorem~\ref{Theorem on divisor problem}:

\begin{cor}
\label{cor:Av Kloost}
For any $q^{\lambda}\leq N\le q$ with $\lambda>0$, there exist constants $k_0$ and $\tau>0$, depending only on $\lambda$ such that
$$
\sum\limits_{1\leq n\leq N}S(mn,a;q) = \sum\limits_{1\leq n\leq N}S(a,mn;q)\ll_{\lambda}Nq^{1/2-\tau}
$$
holds uniformly for any integers $m,a$ satisfying $\gcd(ma,p)=1$ and any $q=p^k$
with $p$ an odd prime, $k\geq k_0$.
\end{cor}

\begin{remark}
Comparing with the  result of Khan~\cite{Khan}, in which the condition $k$ fixed and $p$ sufficiently large is required, Theorem~\ref{Theorem on divisor problem} gives a uniform result for all modulus of the type $q=p^k$ with $p$ an odd prime and $k$ sufficiently large.
\end{remark}

\begin{remark}
In Theorem~\ref{Theorem on Kloosterman sums}, since $\lambda>0$ can be taken arbitrary small, our result shows that Weil bound for sums of Kloosterman sums can be improved on average over a very short interval
for prime power modulus.
\end{remark}

\subsection{Notation} As usual, $\N$, $\Z$, $\R$ and $\Z_p$ are the set of natural numbers, integers, real numbers and $p$-adic integers, respectively.
We use $e(x)$ to denote $e^{2\pi i x}$ and $\fl{x}$ to denote the largest integer not exceeding $x$. For a prime number $p$ and any $n\in\Z$, $p^r\parallel n$ means $p^r|n$ and $p^{r+1}\nmid n$.

For a $p$-adic integer $\alpha\in\Z_p$, denote its $p$-adic order as $v_p (\alpha)$. For a polynomial $f(x)$ with integer coefficients, denote $\ord_p f$ as the $p$-adic order of the largest common divisor of all the coefficients of $f$ (that is, the largest power of $p$ which divides all the coefficients of $f$).

\section{Proof of Theorem~\ref{Theorem on divisor problem}}

We now assume that Theorem~\ref{Theorem on Kloosterman sums} holds,
and then prove it in Section~\ref{sec:KloostSum}.  In particular, here we use
Corollary~\ref{cor:Av Kloost}.

By the definition of $d(n)$, we have
$$
\sum\limits_{\substack{n\leq x\\n\equiv a\bmod q}}d(n)=\sum\limits_{\substack{uv\leq x\\uv\equiv a\bmod q}}1.
$$
Let $\varepsilon>0$ be sufficiently small and $\Delta=1+x^{-2\varepsilon}$. Suppose $U,V$ are parameters of the form $\Delta^{i}$ and $\Delta^{j}$ for $i,j\geq 0$, separately. Then we have
$$
\sum\limits_{\substack{uv\leq x\\uv\equiv a\bmod q}}1=\sum\limits_{U,V}\ \sum\limits_{\substack{uv\leq x\\uv\equiv a\bmod q\\U<u\leq \Delta U\\V<v\leq \Delta V}}1,
$$
where $\sum\limits_{U,V}$ ranges over all the pairs $U= \Delta^i$, $V = \Delta^j$ satisfying $UV\leq x$. The number of these pairs is at most $O\left(x^{4\varepsilon}\log^2 x\right)$.
Removing the condition $uv\leq x$ in the inner sum on the right hand side,
$$
\sum\limits_{\substack{uv\leq x\\uv\equiv a\bmod q}}1=\sum\limits_{U,V}\ \sum\limits_{\substack{uv\equiv a\bmod q\\U<u\leq \Delta U\\V<v\leq \Delta V}}1+O\left(\sum\limits_{\substack{x<n\leq x\Delta^2\\n\equiv a\bmod q}} d(n)\right).
$$
It is obvious that the error term is $O_\varepsilon\left(\frac{x^{1-\varepsilon}}{q}\right)$. We can restrict the range of the sum over  in the first term to $x^{1-2\varepsilon}<UV\leq x$
up to an acceptable error term, since
$$
\sum\limits_{\substack{U,V\\UV\leq x^{1-2\varepsilon}}}\ \sum\limits_{\substack{uv\equiv a\bmod q\\U<u\leq \Delta U\\V<v\leq \Delta V}}1\leq
\sum\limits_{\substack{n\leq x^{1-2\varepsilon}\Delta^2\\n\equiv a\bmod q}}d(n)\ll_\varepsilon\frac{x^{1-\varepsilon}}{q}.
$$
Hence we have
\begin{equation}\label{Expression for sum d(n)}
\sum\limits_{\substack{n\leq x\\n\equiv a\bmod q}}d(n)=\sum\limits_{\substack{U,V\\x^{1-2\varepsilon}\leq UV\leq x}}\
\sum\limits_{\substack{uv\equiv a\bmod q\\U<u\leq \Delta U\\V<v\leq \Delta V}}1+O_\varepsilon\left(\frac{x^{1-\varepsilon}}{q}\right).
\end{equation}

Now we smooth the inner sum over $u$ and $v$. Suppose $f$ and $g$ are smooth functions and compactly supported on the interval $[1,\Delta]$ with derivatives satisfying
$$
\ f^{(j)}\ll_j x^{6j\varepsilon}\quad \text{and}\quad  g^{(j)}\ll_j x^{6j\varepsilon}\quad \text{for any}\  j\geq 0
$$
and $f,g$ equals $1$ in the interval $[1+x^{-6\varepsilon},\Delta-x^{-6\varepsilon}]$. Replacing the $1$ in the inner sum on the right hand side of~\eqref{Expression for sum d(n)} by $f\left(\frac{u}{U}\right)g\left(\frac{v}{V}\right)$,
it is easy to prove that the contribution of the error terms produced in this process can be absorbed by the $O$-term. Then we have
$$
\sum\limits_{\substack{n\leq x\\n\equiv a\bmod q}}d(n)=\sum\limits_{\substack{U,V\\x^{1-2\varepsilon}\leq UV\leq x}}I(U,V;q,a)+O_\varepsilon\left(\frac{x^{1-\varepsilon}}{q}\right),
$$
where $I(U,V;q,a)$ is defined by
$$
I(U,V;q,a):=\sum\limits_{\substack{u,v\\uv\equiv a\bmod q}}f\left(\frac{u}{U}\right)g\left(\frac{v}{V}\right).
$$
By a similar argument, we can get
$$
\frac{1}{\varphi(q)}\sum\limits_{\substack{n\leq x\\\gcd(n,q)=1}}d(n)=\frac{1}{\varphi(q)}\sum\limits_{\substack{U,V\\x^{1-2\varepsilon}\leq UV\leq x}}I(U,V)+O_\varepsilon\left(\frac{x^{1-\varepsilon}}{q}\right),
$$
with $I(U,V)$ given by
$$
I(U,V):=\sum\limits_{\substack{u,v\\\gcd(uv,q)=1}}f\left(\frac{u}{U}\right)g\left(\frac{v}{V}\right).
$$
Thus we have
\begin{align*}
\sum\limits_{\substack{n\leq x\\n\equiv a\bmod q}}&d(n)-\frac{1}{\varphi(q)}\sum\limits_{\substack{n\leq x\\\gcd(n,q)=1}}d(n)\\
=&\sum\limits_{\substack{U,V\\x^{1-2\varepsilon}\leq UV\leq x}}\left(I(U,V;q,a)-\frac{1}{\varphi(q)}I(U,V)\right)+O_\varepsilon\left(\frac{x^{1-\varepsilon}}{q}\right).
\end{align*}
Now by the symmetry of $U$ and $V$, we only need to prove
$$
I(U,V;q,a)-\frac{1}{\varphi(q)}I(U,V)\ll_\varepsilon \frac{x^{1-\varepsilon}}{q}
$$
for any $U$ and $V$ satisfying
$$
x^{1-2\varepsilon}\leq UV\leq x\mand U\leq x^{1/2}.
$$

Thus, we now  fix $U$ and $V$ with this condition.

By the orthogonality of additive characters, we have
$$
I(U,V;q,a)=\frac{1}{q}\sum\limits_{h=1}^{q}e\left(\frac{-ah}{q}\right)\sum\limits_{\substack{u,v\\\gcd(u,q)=1}}
f\left(\frac{u}{U}\right)g\left(\frac{v}{V}\right)e\left(\frac{uvh}{q}\right).
$$
Denote the term for $h=q$ by
$$
\cM:=\frac{1}{q}\sum\limits_{\substack{u\\\gcd(u,q)=1}}
f\left(\frac{u}{U}\right)\sum\limits_{v}g\left(\frac{v}{V}\right).
$$
By the definition of $g$, the inner sum over $v$ is
$$
\sum\limits_{V<v\leq \Delta V}1+O\left(x^{-6\varepsilon}V\right)=(\Delta V-V)+O\left(1+x^{-6\varepsilon}V\right),
$$
which yields
\begin{equation}\label{Expression for mathcal M}
\cM=\frac{1}{q}\sum\limits_{\substack{U<u\leq \Delta U\\\gcd(u,q)=1}}
f\left(\frac{u}{U}\right)(\Delta V-V)+O_\varepsilon\left(\frac{U}{q}+\frac{x^{1-6\varepsilon}}{q}\right).
\end{equation}
Similarly, we have
$$
\frac{1}{\varphi(q)}I(U,V)=\frac{1}{\varphi(q)}\sum\limits_{\substack{U<u\leq \Delta U\\\gcd(u,q)=1}}f\left(\frac{u}{U}\right)
\sum\limits_{\substack{V<v\leq \Delta V\\ \gcd(v,q)=1}}1+O_\varepsilon\left(\frac{x^{1-6\varepsilon}}{q}\right).
$$
To remove the condition $\gcd(v,q)=1$ in the sum over $v$, we use the formula
$$
\sum\limits_{d|n}\mu(d)=\left\{
\begin{array}{ll}
1,\quad &\text{if}\ n=1,\\
0,\quad &\text{otherwise},
\end{array}
\right.
$$
and get
$$
\sum\limits_{\substack{V<v\leq \Delta V\\ \gcd(v,q)=1}}1=
\sum\limits_{d|q}\mu(d)\sum\limits_{\substack{V<v\leq \Delta V\\d|v}}1.
$$
It follows that
$$
\sum\limits_{\substack{V<v\leq \Delta V\\ \gcd(v,q)=1}}1=
\frac{\varphi(q)}{q}\left(\Delta V-V\right)+O_\varepsilon(q^{\varepsilon}),
$$
where we used
$$
\sum\limits_{d|q}\frac{\mu(d)}{d}=\frac{\varphi(q)}{q}.
$$
Thus we obtain
\begin{equation}\label{Expression for I(U,V)}
\begin{split}
\frac{1}{\varphi(q)}I(U,V)=\frac{1}{q}\sum\limits_{\substack{U<u\leq \Delta U\\\gcd(u,q)=1}}f\left(\frac{u}{U}\right)&(\Delta V-V)\\
&+O_\varepsilon\left(\frac{Uq^{\varepsilon}}{\varphi(q)} +\frac{x^{1-6\varepsilon}}{q}\right).
\end{split}
\end{equation}
Recall that $U\leq x^{1/2}$, then for sufficiently small $\varepsilon$, we get
$$
\cM-\frac{1}{\varphi(q)}I(U,V)\ll_\varepsilon\frac{x^{1-6\varepsilon}}{q}
$$
from~\eqref{Expression for mathcal M} and~\eqref{Expression for I(U,V)}.
Now we only need to estimate the sum
$$
\cE:=\frac{1}{q}\sum\limits_{h=1}^{q-1}e\left(\frac{-ah}{q}\right)\sum\limits_{\substack{u,v=-\infty\\\gcd(u,q)=1}}^\infty
f\left(\frac{u}{U}\right)g\left(\frac{v}{V}\right)e\left(\frac{uvh}{q}\right)
$$
and show that there exists an absolute constant $\sigma>0$ such that
\begin{equation}\label{Upper bound for for mathcal E}
\cE \ll \frac{x}{q^{1+\sigma}}
\end{equation}
holds uniformly for $q\leq x^{2/3+\sigma}$. Note that since the functions $f$ and $g$ are compactly supported, the sum over $u$ and $v$ is
actually finite.

Noting $q=p^k$ with $p$ an odd prime, write
$$
\cE=\frac{1}{q}\sum\limits_{0\leq r<k}\sum\limits_{\substack{1\leq h\leq p^k\\p^r\parallel h}}e\left(\frac{-ah}{q}\right)
\sum\limits_{\substack{u,v=-\infty\\\gcd(u,q)=1}}^\infty
f\left(\frac{u}{U}\right)g\left(\frac{v}{V}\right)e\left(\frac{uvh}{q}\right).
$$
It follows that
$$
\cE=\frac{1}{q}\sum\limits_{0\leq r<k}\ \sideset{}{^*}\sum\limits_{b\bmod p^{k-r}}e\left(\frac{-ab}{p^{k-r}}\right)
\sum\limits_{\substack{u,v=-\infty\\\gcd(u,q)=1}}^\infty
f\left(\frac{u}{U}\right)g\left(\frac{v}{V}\right)e\left(\frac{uvb}{p^{k-r}}\right).
$$
The inner sum for $u,v$ can be written as
$$
\mathcal{F}:=\sum\limits_{\substack{s,t\bmod p^{k-r}\\\gcd(t,p)=1}}e\left(\frac{stb}{p^{k-r}}\right)\sum\limits_{\substack{u\equiv t\bmod(p^{k-r})}}f\left(\frac{u}{U}\right)
\sum\limits_{v\equiv s\bmod(p^{k-r})}g\left(\frac{v}{V}\right).
$$
Applying Poisson summation (see~\cite[Lemma~2.1]{FKM}), it equals to
$$
\frac{UV}{p^{2(k-r)}}\sum\limits_{\substack{s,t\bmod p^{k-r}\\\gcd(t,p)=1}}e\left(\frac{stb}{p^{k-r}}\right)
\sum\limits_{m,n}e\left(\frac{sn+tm}{p^{k-r}}\right)\widehat{f}\left(\frac{nU}{p^{k-r}}\right)\widehat{g}\left(\frac{mV}{p^{k-r}}\right).
$$
Summing over $s$, we get
$$
\mathcal{F}=\frac{UV}{p^{k-r}}\sum\limits_{\substack{m,n\\  \gcd(n,p)=1}}e\left(-\frac{mn\overline{b}}{p^{k-r}}\right)
\widehat{f}\left(\frac{nU}{p^{k-r}}\right)\widehat{g}\left(\frac{mV}{p^{k-r}}\right),
$$
which gives
$$
\cE=\sum\limits_{0\leq r<k}\frac{UV}{p^{2k-r}}\sum\limits_{\substack{m,n\\ \gcd(n,p)=1}}
\widehat{f}\left(\frac{nU}{p^{k-r}}\right)\widehat{g}\left(\frac{mV}{p^{k-r}}\right)S(a,mn;p^{k-r}).
$$
By partial integration, the sums over $m$ and $n$ can be restricted to
\begin{equation}\label{Range of m,n}
|n|\leq \frac{x^{\varepsilon}p^{k-r}}{U}, \qquad  |m|\leq \frac{x^{\varepsilon}p^{k-r}}{V},
\end{equation}
up to an error term $O\(x^{-100}\)$.

Break the sum over $r$ into two sums
\begin{equation}\label{Breaking the sum over r}
\cE=\cE_1 + \cE_2,
\end{equation}
where
\begin{align*}
 \cE_1&= \sum\limits_{0\leq r<k/8}\frac{UV}{p^{2k-r}}\sum\limits_{\substack{m,n\\\gcd(n,p)=1}}
\widehat{f}\left(\frac{nU}{p^{k-r}}\right)\widehat{g}\left(\frac{mV}{p^{k-r}}\right)S(a,mn;p^{k-r}),\\
 \cE_2&= \sum\limits_{k/8\leq r<k} \frac{UV}{p^{2k-r}}\sum\limits_{\substack{m,n\\\gcd(n,p)=1}}
\widehat{f}\left(\frac{nU}{p^{k-r}}\right)\widehat{g}\left(\frac{mV}{p^{k-r}}\right)S(a,mn;p^{k-r}).
\end{align*}

For large $r$, we apply the Weil bound for Kloosterman sums (see~\cite[Corollary~11.12]{IwKow}) and derive
\begin{equation}\label{Contribution of large r}
 \cE_2\ll x^{2\varepsilon}p^{k/2}\sum\limits_{k/8\leq r<k}p^{-3r/2}\ll x^{2\varepsilon}p^{5k/16},
\end{equation}
which is small enough.

Now we only need to bound $\cE_1$.
Note that for $\gcd(a,p)=1$ and $p^j\parallel m$,
$$
S(a,m;p^{k})=\left\{
\begin{array}{ll}
\mu(p^k),\quad &{\text{if}}\ j\geq k,\\
0,\quad  &{\text{if}}\ 0<j<k.
\end{array}
\right.
$$
We have
$$
 \cE_1=\sum\limits_{0\leq r<k/8}\frac{UV}{p^{2k-r}}\sum\limits_{\substack{m,n\\ \gcd(mn,p)=1}}
\widehat{f}\left(\frac{nU}{p^{k-r}}\right)\widehat{g}\left(\frac{mV}{p^{k-r}}\right)S(a,mn;p^{k-r}).
$$
Our cancellation comes from the sum over $n$. By~\eqref{Range of m,n}, we only deal with
$$
\mathcal{G}:=\sum\limits_{\substack{1\leq n\leq \frac{x^{\varepsilon}p^{k-r}}{U}\\ \gcd(mn,p)=1}}\widehat{f}\left(\frac{nU}{p^{k-r}}\right)S(a,mn;p^{k-r}).
$$
The contribution of the part $n\leq -1$ can be treated similarly.
Denote the sums over
 $1\leq n\leq q^{1/10}$ and $n  > q^{1/10}$ by $\mathcal{G}_{n\leq q^{1/10}}$ and $\mathcal{G}_{n> q^{1/10}}$, respectively. Then Weil bound for Kloosterman sums  (see~\cite[Corollary~11.12]{IwKow}) gives
$$
\mathcal{G}_{n\leq q^{1/10}}\ll q^{1/10}(k-r+1)p^{\frac{k-r}{2}}.
$$
Denote the contribution of $\mathcal{G}_{n\leq q^{1/10}}$ to $\cE_1$ by $\mathfrak{C}_1$, then we have
\begin{equation}\label{Contribution of n small}
\mathfrak{C}_1\ll (k+1)x^{1/2}q^{1/10}\sum\limits_{0\leq r<k/8}p^{-\frac{k+r}{2}}\ll_\varepsilon x^{1/2}q^{-2/5}\mathcal{L}^2,
\end{equation}
which is acceptable. For $\mathcal{G}_{n>q^{1/10}}$, it follows from partial summation that
\begin{align*}
\mathcal{G}_{n>q^{1/10}}=\widehat{f}\left(q^{1/10}\right)&
\sum\limits_{q^{1/10}< n\leq x^{\varepsilon}p^{k-r}/U}S(a,mn;p^{k-r})\\
&+\int_{q^{1/10}}^{\frac{x^{\varepsilon}p^{k-r}}{U}}\sum\limits_{n\leq t}S(a,mn;p^{k-r})\left(\widehat{f}\left(\frac{tU}{p^{k-r}}\right)\right)^{\prime}dt.
\end{align*}
Note that $\left|\widehat f(t)^{\prime}\right|\leq 1$ for any $t\in \R$, then
$$
\left(\widehat{f}\left(\frac{tU}{p^{k-r}}\right)\right)^{\prime}\ll \frac{U}{p^{k-r}}.
$$

Now by Corollary~\ref{cor:Av Kloost}
there exists a constant $\rho>0$ (which does not depend on $\varepsilon$), such that
$$
\mathcal{G}_{n> q^{1/10}}\ll_\varepsilon \frac{x^{2\varepsilon}p^{k-r}q^{1/2-\rho}}{U}.
$$
Let $\mathfrak{C}_2$ denote the contribution of $\mathcal{G}_{n>q^{1/10}}$ to $\cE_1$, then we have
\begin{equation}\label{Contribution of n large}
\mathfrak{C}_2\ll_\varepsilon\sum\limits_{0\leq r<k/8}x^{3\varepsilon}q^{1/2-\rho}
p^{-r}\ll_\varepsilon x^{3\varepsilon}q^{1/2-\rho}.
\end{equation}
Since $\varepsilon$ is arbitrary, combining~\eqref{Breaking the sum over r}, \eqref{Contribution of large r}, \eqref{Contribution of n small}
and~\eqref{Contribution of n large} with~\eqref{Upper bound for for mathcal E}, we complete the proof of Theorem~\ref{Theorem on divisor problem}.

\section{Proof of Theorem~\ref{Theorem on Kloosterman sums}}
\label{sec:KloostSum}
\subsection{Preparations}

We start with the following well-known elementary statement.

\begin{lemma}\label{Basic lemma on P-adic order}
Let $p$ be a prime number and $n\in\N$, then we have
$$
\ord_p (n!)=\sum\limits_{j=1}^{\infty}\fl{\frac{n}{p^j}}.
$$
\end{lemma}

We also need the following technical result.

\begin{lemma}\label{Upper bound for p-adic order}
For every integer $i\geq 1$, let
$$
\binom{1/2}{i}:=\frac{1/2(1/2-1)\cdots(1/2-i+1)}{i!}
$$
and
$3\leq u\leq i$, then we have
$$
\nu_p\left(\binom{1/2}{i}\frac{i!}{(i-u)!}\right)\leq u\sum\limits_{j=1}^{\infty}\frac{1}{p^j}+E(i,p),
$$
with $$
|E(i,p)|\leq \frac{3\log(2i)}{\log p}.
$$
\end{lemma}

\begin{proof}
Noting that
$$
\binom{1/2}{i}=\frac{(-1)^{i-1}(2i-3)!}{2^{2i-2}i!(i-2)!}\qquad  {\text{for}}\  i\geq 3,
$$
we have
$$
\binom{1/2}{i}\frac{i!}{(i-u)!}=(-1)^{i-1}\frac{(2i-3)!}{2^{2i-2}(i-2)!(i-u)!}
$$
for $i\geq 3$, then by Lemma~\ref{Basic lemma on P-adic order}, we have
$$
\nu_p\left(\frac{(2i-3)!}{(i-2)!(i-u)!}\right)=\sum\limits_{j=1}^{\infty}\left(\fl{\frac{2i-3}{p^j}}-\fl{\frac{i-2}{p^j}}-\fl{\frac{i-u}{p^j}}\right).
$$
Terms in the above sum vanish when $j>J$, where
$$
J = \frac{\log(2i-3)}{\log p},
$$
which yields
$$
\nu_p\left(\frac{(2i-3)!}{(i-2)!(i-u)!}\right)=(u-1)\sum\limits_{j=1}^{J}\frac{1}{p^j}+E(i,p),
$$
with $|E(i,p)|\leq 3 J \le 3\log(2i)/\log p$. Then the result follows from extending the range of the summation.
\end{proof}

\begin{lemma}\label{Special Kloosterman sums equals to 0}
Let $p$ be an odd prime and $k\geq 2$ be a positive integer. If $(a,p)=1$ and $p|n$, then the Kloosterman sums $S(n,a;p^k)=0$.
\end{lemma}
\begin{proof}
By assumption, we may suppose $n=p^rm$ with $(m,p)=1$ and $r\geq 1$. If $r\geq k$, then $S(n,a;p^k)=S(0,a;p^k)$ is a Ramanujan sum and equals to $0$,
since $(a,q)=1$ and $k\geq 2$. If $1\leq r<k$, noting
$$
S(n,a;p^k)=\sideset{}{^*}\sum\limits_{{b \bmod p^k}} e\left(\frac{p^r m \overline{b} + ab}{p^k}\right),
$$
we have
$$
S(n,a;p^k)=\sideset{}{^*}\sum\limits_{{y \bmod p^{k-r}}} \sum_{x \bmod p^r} e\left(\frac{p^r m\ \overline{y + p^{k-r} x} + a(y + p^{k-r} x)}{p^k}\right).
$$
Summing over $x$, we get
$$
S(n,a;p^k)=\sideset{}{^*}\sum\limits_{{y \bmod p^{k-r}}} e\left(\frac{p^r m y^{-1}+ay}{p^k}\right) \sum_{x \bmod p^r} e\left(\frac{a x}{p^r}\right)=0,
$$
which concludes the proof.
\end{proof}

Let $\Re\, z$ denote the real part of a complex number  $z$.
\begin{lemma}\label{Expression for Kloosterman sums with prime power moduli}
For $\gcd(a,q)=1$ and $q=p^k$ with $k\geq 2$, we have
$$
S(n,a;q)=\left\{
\begin{array}{ll}
2\left(\frac{l}{p}\right)^kq^{1/2}\Re\, \vartheta_q e\left(\frac{2l}{q}\right),\quad  &{\text{if}}\left(\frac{na}{p}\right)=1,\\
0,\quad &{\text{if}}\left(\frac{na}{p}\right)=-1,
\end{array}
\right.
$$
where $l^2\equiv na \bmod q$, $\left(\frac{l}{p}\right)$ is the Legendre symbol, $\vartheta_q$ equals $1$ if $q\equiv1 \bmod 4$ and $i$ if $q\equiv3 \bmod 4$.
\end{lemma}
\begin{proof}
This is~\cite[Equation~(12.39)]{IwKow}.
\end{proof}

\begin{lemma}\label{Main lemma}
Suppose that $d,\mu\in \N$, $d\geq 300$, $ \mu\geq d+1$, $ \beta=\fl{\mu/10}+1$, $ f(X)=a_1X+\cdots+a_{d+1}X^{d+1}\in \Z[X]$. Let $r$ be defined by the relation $P^r=p^{\mu}$ and $\mu\log p>10^{8}rd\log d$. Then if $1\leq r\leq d/300$, there exists an absolute constant $c>10^{-13}$, such that
$$
\left|\sum\limits_{1\leq x\leq P}e\left(\frac{f(x)}{p^{\mu}}\right)\right|\leq 3P^{1-c/r^{2}}+nR,
$$
where $R$ is the maximum number of solutions of the congruence
$$
f^{(u)}(x)\equiv 0\ \bmod p^{\beta},\indent 1\leq x\leq P,
$$
for $25r\leq u\leq 27r$.
\end{lemma}
\begin{proof}
This is~\cite[Theorem~2]{Korobov}.
\end{proof}

\begin{lemma}\label{Konyakin's lemma}
Suppose $f(X)=a_0+a_1X+\cdots+a_dX^d\in \Z[X]$ with the coefficients satisfying $\gcd(a_0,..., a_d,m)=1$.
Let $\rho(f,m)$ be the number of solutions of the congruence
$$
f(x)\equiv 0\ \bmod m.
$$
Then for $d\geq 2$, we have
$$
\rho(f,m)\leq c_dm^{1-1/d},
$$
where $c_d=d/e+O(\log^2 d)$ with $e$ being the base of the natural logarithm.
\end{lemma}

\begin{proof}
This is  the main result of~\cite{Konyagin}.
\end{proof}

\begin{lemma}\label{Stepannov and Shparlinski's lemma}
Let $Q,\mu$ be positive integers, $p$ be a prime number.
Suppose $f(X)=a_0+a_1X+\cdots+a_dX^d\in \Z[X]$ with the coefficients satisfying $\gcd(a_0,..., a_d,p)=1$.
 Then for the number of solutions $R(Q,p^{\mu})$ of the congruence
$$
f(x)\equiv 0\ \bmod p^{\mu},\indent 1\leq x\leq Q,
$$
the estimate
$$
R(Q,p^{\mu})\ll d(Qp)^{1-1/d}+dQp^{-\mu/d}
$$
holds , where the implied constant in $\ll$ is absolute.
\end{lemma}
\begin{proof}
By Lemma~\ref{Konyakin's lemma}, we have $R(p^{\mu},p^{\mu})\ll dp^{\mu-\mu/d}$. Then for $Q\geq p^{\mu}$,
$$
R(Q,p^\mu)\leq R(p^{\mu},p^\mu)\left(\frac{Q}{p^\mu}+1\right)\ll dQp^{-\mu/d}.
$$
If $Q<p^{\mu}$, there exists a unique non-negative integer $\omega$ such that $p^{\omega-1}<Q\leq p^\omega$. It is clear that $\omega\leq \mu$ and $p^{\omega}\leq pQ$, which yields
$$
R(Q,p^{\mu})\leq R(Q,p^{\omega})\leq R(p^{\omega},p^{\omega})\ll dp^{\omega-\omega/d}\leq d(Qp)^{1-1/d}.
$$
Now the result follows from the above two estimates.
\end{proof}

\begin{remark}
We remark that for a fixed $d$, Konyagin and Steger~\cite{KoSte} give stronger estimates on $R(Q,p^{\mu})$
than that of Lemma~\ref{Stepannov and Shparlinski's lemma},
but we prefer to us to keep the dependence on $d$ explicit. This maybe useful if one needs to derive a version
of  Theorem~\ref{Theorem on Kloosterman sums} with $\lambda$ which is a slowly decreasing function of $q$.
\end{remark}

\subsection{Concuding the proof}

Let $\gcd(a,p)=1$, $q=p^k$ with $p$ an odd prime and $k\geq 2$ a positive integer. For a given $\lambda>0$, we may suppose $10\leq q^{\lambda}\leq N\leq q$ without loss of generality, and consider the upper bound of the sum
$$
S(N):=\sum\limits_{1\leq n\leq N}S(mn,a;q).
$$

Take
\begin{equation}
\label{Def of s}
s:=\fl{\frac{\log N}{B\log p}}
\end{equation}
with a sufficiently large constant  $B >0$ (depending on $\lambda$) and $T:=\fl{N/p^s}$.

Then Weil bound for Kloosterman sums  (see~\cite[Corollary~11.12]{IwKow}) gives
$$
S(N)=S(p^{s}T)+O(p^sd(q)q^{1/2}).
$$
When $q^{\lambda}\leq N\leq q$, the $O$-term can be estimated trivially as
$$
p^sd(q)q^{1/2}\ll qN^{1/B}\ll Nq^{1/2-(1-1/B)\lambda},
$$
which is small enough, hence we only need to bound $S(p^{s}T)$. By Lemma~\ref{Special Kloosterman sums equals to 0}, the sum over $n$ with $p|n$ vanishes, thus
$$
S(p^sT)=\sum\limits_{\substack{1\leq n\leq p^sT\\ \gcd(p,n)=1}}S(mn,a;q).
$$
Now we apply Lemma~\ref{Expression for Kloosterman sums with prime power moduli}. Since there are two solutions for the quadratic congruence of $l$,
it's necessary to note that the expression for Kloosterman sums doesn't depend on which solution we choose. Hence we may write
$$
S(p^sT)=q^{1/2}\sum\limits_{\substack{n\leq p^sT\\\left(\frac{nma}{p}\right)=1}}\sum\limits_{l^2\equiv nma \bmod q}\left(\frac{l}{p}\right)^k\Re\, \vartheta_q e\left(\frac{2l}{q}\right),
$$
where $\sum\limits_{l^2\equiv nma \bmod q}$ means summing over the two solutions of the congruence $l^2\equiv nma \bmod q$.
Classify $n$ by the remainder of $mna\bmod p^s$,
\begin{equation}\label{Expression 1 for S(p^sT)}
S(p^sT)=q^{1/2}\sum\limits_{\substack{1\leq\alpha<p^s\\\left(\frac{\alpha}{p}\right)=1}}
\sum\limits_{\substack{n\leq p^sT\\nma\equiv\alpha \bmod p^s}}
\sum\limits_{l^2\equiv nma \bmod q}\left(\frac{l}{p}\right)^k\Re\, \vartheta_q e\left(\frac{2l}{q}\right).
\end{equation}

To solve the quadratic congruence in the inner sum, we use the following argument, which is similar to that in~\cite{Khan}. Since $(ma,q)=1$, suppose $ma\xi\equiv1 \bmod q$ and $\vartheta\equiv\xi \bmod p^s$ with $1\leq\vartheta<p^s,s\geq 1$.
From $nma\equiv\alpha\bmod p^s$, we have $n\equiv\vartheta\alpha\bmod p^s$, which implies
that there exists $t\in \Z$, such that $n=\vartheta\alpha+p^st$. Now we have
$$
l^2\equiv ma(\vartheta\alpha+p^st)\equiv ma\vartheta\alpha(1+\kappa p^st)\ \bmod q,
$$
with $\vartheta\alpha\kappa\equiv1\bmod q$. Note that $ma\vartheta\equiv1\bmod p$, then $\left(\frac{ma\vartheta\alpha}{p}\right)=1$. By Hensel's lemma, there exists $\omega\in\Z$, such that
$\omega^2\equiv ma\vartheta\alpha \bmod q.$ Thus
$$
l^2\equiv nma\equiv \omega^2(1+\kappa p^st)\ \bmod q.
$$
We remark that $\omega$ is determined by $m,a,\alpha,p^s$ and does not depend on $n$.
Consider $1+\kappa p^st$ in the $p$-adic field $\Q_p$. By Taylor's expansion
(see~\cite[Chapter~IV.1]{Koblitz}), we have
$$
(1+\kappa p^st)^{1/2}=1+\sum\limits_{i=1}^{\infty}\binom{1/2}{i}\kappa^ip^{is}t^i,
$$
for $s\geq 1$. Here the coefficients $\binom{1/2}{i}=\frac{1/2(1/2-1)\cdots(1/2-i+1)}{i!}$ with $i\geq 1$ happen to be $p$-adic integers, since $p$ is an odd prime.
Then we have
$$
(1+\kappa p^st)^{1/2}\equiv\sum\limits_{i=0}^{\fl{k/s}}g(i)\kappa^ip^{is}t^i\ \bmod p^k,
$$
where $g(0)=1$ and $g(i)$ with $1\leq i\leq \fl{k/s}$ are integers given by
\begin{equation}\label{Definition of g(i)}
g(i)\equiv\binom{1/2}{i}\bmod p^k,\qquad  0\leq g(i)<p^k.
\end{equation}
Thus we get two solutions for the quadratic congruence of $l$ in the inner sum
of~\eqref{Expression 1 for S(p^sT)}.
$$
l\equiv\pm\omega f(t)\bmod q,
$$
where
\begin{equation}\label{Definition of f(t)}
f(t):=\sum\limits_{i=0}^{ \fl{k/s}}g(i)\lambda^ip^{is}t^i.
\end{equation}
Choosing the solution $l\equiv\omega f(t)\bmod q$ and noting that $f(t)\equiv g(0)\equiv 1\bmod p$, we have
$$
S(p^sT)=2q^{1/2}\sum\limits_{\substack{1\leq\alpha<p^s\\\left(\frac{\alpha}{p}\right)=1}}\left(\frac{\omega}{p}\right)^k\sum\limits_{\substack{t\leq \frac{p^sT-\vartheta\alpha}{p^s}}} \Re\,\vartheta_q e\left(\frac{2\omega f(t)}{q}\right),
$$
which gives
$$
S(p^sT)\leq 2q^{1/2}\sum\limits_{\substack{1\leq\alpha<p^s\\\left(\frac{\alpha}{p}\right)=1}}
\left|\sum\limits_{\substack{t\leq \frac{p^sT-\vartheta\alpha}{p^s}}}e\left(\frac{2\omega f(t)}{q}\right)\right|.
$$
Recalling $1\leq\vartheta<p^s$, we have
\begin{equation}\label{Pre-bound of S(p^sT)}
S(p^sT)\leq 2q^{1/2}\sum\limits_{\substack{1\leq\alpha<p^s\\\left(\frac{\alpha}{p}\right)=1}}\left|\sum\limits_{\substack{t\leq T}}e\left(\frac{2\omega f(t)}{q}\right)\right|+O\left(q^{1/2}p^{2s}\right).
\end{equation}
Since $B>0$ in~\eqref{Def of s} is fixed and sufficiently large and $q^{\lambda}\leq N\leq q$, the contribution of the above $O$-term is
$$
q^{1/2}p^{2s} \ll q^{1/2}N^{2/B} \ll  Nq^{1/2-(1-2/B)\lambda},
$$
which is small enough. Hence we only need to deal with the first term in~\eqref{Pre-bound of S(p^sT)}. Denote the inner sum over $t$ as
$$
M:=\sum\limits_{\substack{t\leq T}}e\left(\frac{2\omega f(t)}{q}\right).
$$
Applying Lemma~\ref{Main lemma} to $M$, we obtain
\begin{equation}
\label{Bound for M}
|M|\leq 3T^{1-c/r^2}+dR.
\end{equation}
Here $c>10^{-13}$ is an absolute constant, $r$ is given by $T^r=p^k$, $d:= \fl{k/s}$ is the degree of $f(t)$ and $R$ is the maximal number of solutions of the congruences
\begin{equation}\label{Congruence equation}
f^{(u)}(x)\equiv 0\bmod p^\beta,\qquad  1\leq x\leq T,
\end{equation}
for $25r\leq u\leq 27r$, where $\beta:=\fl{k/10}+1$. Note that
\begin{equation}\label{Lower bound for T}
T=\fl{N/p^s}\geq \fl{N^{1-1/B}} \geq \fl{p^{(1-1/B)\lambda k}} \geq p^{(1-1/B)\lambda k/2}.
\end{equation}
Recall $T^r=p^k$, then
\begin{equation}\label{Upper bound for r}
r=\frac{k\log p}{\log T}\leq \frac{2}{(1-1/B)\lambda}.
\end{equation}
Let $\mathfrak{F}_1$ denote the contribution of the term $3T^{\left(1-\frac{c}{r^2}\right)}$ in~\eqref{Bound for M} to $S(p^sT)$, then
$$
\mathfrak{F}_1 \ll q^{1/2}p^sT^{1-c/r^2}\ll Nq^{1/2}T^{-c/r^2}\ll Nq^{1/2-\delta_1(\lambda)},
$$
where
$$
\delta_1(\lambda):=\frac{c\lambda^3(1-1/B)^3}{8}.
$$

Now we estimate the contribution of $dR$ in~\eqref{Bound for M} to $S(p^sT)$.
To this aim, we give the upper bound for $d=k/s$ first, which is
\begin{equation}\label{Upper bound for d}
d\leq k/s=\frac{k}{\fl{ \frac{\log N}{B\log p}}}\leq \frac{k}{\left[\lambda k/B\right]}\leq \frac{2B}{\lambda},
\end{equation}
provided
\begin{equation}\label{Condition 1 of k}
k\geq \frac{10B}{\lambda}.
\end{equation}
Let $R_u$ denote the number of solutions of the equation~\eqref{Congruence equation}, then
$$
R=\max\limits_{25r\leq u\leq 27r}R_u.
$$
From Lemma~\ref{Stepannov and Shparlinski's lemma}, we have
$$
dR\ll d^2\max\limits_{25r\leq u\leq 27r}(pT)^{1-1/(d-u)}+d^2T\max\limits_{25r\leq u\leq 27r}p^{-(\beta-\ord_pf^{(u)})/(d-u)},
$$
which yields
\begin{equation}\label{Upper bound of dR}
dR\ll d^2p\max\limits_{25r\leq u\leq 27r}T^{1-1/d}+d^2T\max\limits_{25r\leq u\leq 27r}p^{-(\beta-\ord_pf^{(u)})/d}.
\end{equation}
Let $\mathfrak{F}_2$ denote the contribution of the first term on the right hand side to $S(p^sT)$. Then
$$
\mathfrak{F}_2\ll q^{1/2}d^2p^{s+1}T^{1-1/d}\ll Nq^{1/2}d^2pT^{-1/d}.
$$
Further, using the lower bound~\eqref{Lower bound for T} of $T$ and the upper
bound~\eqref{Upper bound for d} of $d$,
$$
\mathfrak{F}_2\ll\frac{4B^2}{\lambda^2}q^{-(\lambda^2(B-1)/4+1/k}\leq\frac{4B^2}{\lambda^2}q^{-\lambda^2(B-1)/8},
$$
provided
\begin{equation}\label{Condition 2 of k}
k\geq \frac{8}{\lambda^2(B-1)}.
\end{equation}

Now only the contribution of the second term in~\eqref{Upper bound of dR} to $S(p^sT)$ is left. Let's estimate the upper bound of $\ord_pf^{(u)}$ for $25r\leq u\leq 27r$.
Noting that $(\lambda,p)=1$ in the definition~\eqref{Definition of f(t)} of $f(t)$, we have
$$
\ord_p(f^{(u)})=\min\limits_{\substack{u\leq i\leq d\\g(i)\neq 0}}\left(\nu_p\left(g(i)\frac{i!}{(i-u)!}\right)+is\right).
$$
We claim that if $k$ is sufficiently large, then
$$
g(i)\neq 0 \mand  \nu_p(g(i))=\nu_p\left(\binom{1/2}{i}\right),
$$
for all $u\leq i\leq d$. To see this, recall
$$
\binom{1/2}{i}=\frac{(-1)^{i-1}(2i-3)!}{2^{2i-2}i!(i-2)!} \qquad  \text{for}\  i\geq 3
$$
which is an $p$-adic integer. Then an argument similar to that in the proof of Lemma~\ref{Upper bound for p-adic order} gives
$$
\nu_p\left(\binom{1/2}{i}\right)=-\sum\limits_{j=1}^{\frac{\log(2i-3)}{\log p}}\frac{1}{p^j}+E^{\prime}(i,p),
$$
with $|E^{\prime}(i,p)|\leq \frac{3\log(2i)}{\log p}$. Therefore, for $u\leq i\leq d$, we have
$$
\nu_p\left(\binom{1/2}{i}\right)\leq \frac{3\log(2d)}{\log p}\leq \frac{3\log(4B/\lambda)}{\log 3},
$$
which implies
$$
\nu_p\left(\binom{1/2}{i}\right)\leq k-1
$$
provided
\begin{equation}\label{Condition 3 of k}
k\geq \frac{3\log(4B/\lambda)}{\log 3}+1.
\end{equation}
Now our claim follows from the definition~\eqref{Definition of g(i)} of $g(i)$. Thus we can remove the condition $g(i)\neq 0$ for $k$ satisfying the above condition and get
$$
\ord_p(f^{(u)})=\min\limits_{u\leq i\leq d}\left(\nu_p\left(\binom{1/2}{i}\frac{i!}{(i-u)!}\right)+is\right).
$$
By Lemma~\ref{Upper bound for p-adic order}, we have
$$
\ord_p(f^{(u)})\leq \min\limits_{u\leq i\leq d}\left(u\sum\limits_{j=1}^{\infty}\frac{1}{p^j}+\frac{3\log(2i)}{\log p}+is\right),
$$
which yields
$$
\ord_p\left(f^{(u)}\right)\leq u\sum\limits_{j=1}^{\infty}\frac{1}{p^j}+\frac{3\log(2u)}{\log p}+us.
$$
Hence, for every $25r\leq u\leq 27r$, we have an uniform bound
$$
\ord_p\left(f^{(u)}\right)\leq \frac{27r}{p-1}+3\log(54r)+27rs.
$$

Let $\mathfrak{F}_3$ denote the contribution of the second term in~\eqref{Upper bound of dR} to $S(p^sT)$, then
$$
\mathfrak{F}_3\ll \left(4B^2/\lambda^2\right)Nq^{1/2}\max\limits_{25r\leq u\leq 27r}p^{-(\beta-\ord_pf^{(u)}/{d}}
$$
by~\eqref{Pre-bound of S(p^sT)}, \eqref{Bound for M} and~\eqref{Upper bound for d}.
Recall $\beta=[k/10]+1$ and by~\eqref{Upper bound for d} again,
$$
\frac{\beta-\ord_p\left(f^{(u)}\right)}{d}\geq \frac{\lambda}{2B}\left(k/10-\ord_p\left(f^{(u)}\right)\right),
$$
which gives
$$
\frac{\beta-\ord_p\left(f^{(u)}\right)}{d}\geq \frac{\lambda}{2B}\left(k/10-3\log(54r)-54rs\right),
$$
Note that $N\leq p^k$, then, recalling~\eqref{Def of s}, we obtain
$$
s=\left[\frac{\log N}{B\log p}\right]\leq \frac{\log p^k}{B\log p}\leq \frac{k}{B}.
$$
From this and the upper bound~\eqref{Upper bound for r} of $r$, we get
$$
\frac{\beta-\ord_pf^{(u)}}{d}\geq \frac{\lambda}{2B}\left(k/10-\frac{108k}{\lambda(B-1)}-3\log\left(\frac{108}{\lambda(1-1/B)}\right)\right).
$$
Taking $B=B(\lambda)>0$ sufficiently large, such that
$$
\frac{108}{\lambda(B-1)}\leq 1/20.
$$
Then
$$
\frac{\beta-\ord_pf^{(u)}}{d}\geq \frac{\lambda}{2B}\left(k/20-3\log\left(\frac{108}{\lambda(1-1/B)}\right)\right).
$$
It follows that
$$
\frac{\beta-\ord_pf^{(u)}}{d}\geq \frac{\lambda k}{80B}
$$
provided
\begin{equation}\label{Condition 4 of k}
k\geq 120\log\left(\frac{108}{\lambda(1-1/B)}\right),
\end{equation}
which yields
$$
\mathfrak{F}_3\ll\left(4B^2/\lambda^2\right)Nq^{1/2-\lambda k/(80B)}.
$$
We now choose $k_0$ in such a way that for $k \ge k_0$ the conditions~\eqref{Condition 1 of k},
 \eqref{Condition 2 of k}, \eqref{Condition 3 of k} and~\eqref{Condition 4 of k} are satisfied,
this completes the proof of Theorem~\ref{Theorem on Kloosterman sums}.

\section*{Acknowledgement}

The first two authors gratefully acknowledge the support,   hospitality
and   excellent conditions of the School of Mathematics and Statistics of UNSW during their visit.

This work was supported by NSFC Grant 11401329 (for K. Liu), by ARC Grant DP140100118 (for I. E. Shparlinski), by NSFC Grant 11201275
and the Fundamental Research Funds for the Central Universities
Grant~GK201503014 (for T. P. Zhang).

\end{document}